\documentclass[a4paper]{amsart}

\usepackage{latexsym}
\usepackage{amscd}
\usepackage{graphics}
\usepackage[pdftex]{graphicx}
\usepackage{amsmath}
\usepackage{amssymb}
\usepackage{mathrsfs}
\usepackage{amsthm}
\usepackage{bbm}
\usepackage{color}
\usepackage{accents}
\usepackage{enumerate}
\input xy
\xyoption{all}

\newcommand{\N}{\mathbb{N}}
\newcommand{\Z}{\mathbb{Z}}

\newcommand{\R}{\mathbb{R}}
\newcommand{\C}{\mathbb{C}}

\newcommand{\T}{\mathbb{T}}

\newcommand{\D}{\mathbb{D}}

\newcommand{\supp}{\mathrm{supp\,}}             

\theoremstyle{plain}
\newtheorem*{theorem*}{Theorem}
\newtheorem{theorem}{Theorem}
\newtheorem{proposition}[theorem]{Proposition}

\theoremstyle{definition}

\theoremstyle{remark}

\title[Contact forms with large systolic ratio in dimension three]{Contact forms with large systolic ratio in dimension three}
\author[Abbondandolo]{Alberto Abbondandolo}
\author[Bramham]{Barney Bramham}
\author[Hryniewicz]{Umberto L. Hryniewicz}
\author[Salom\~ao]{Pedro A. S. Salom\~ao}

\begin{document}

\maketitle

\begin{abstract}
The systolic ratio of a contact form on a closed three-manifold is the quotient of the square of the shortest period of closed Reeb orbits by the contact volume. We show that every co-orientable contact structure on any closed three-manifold is defined by a contact form with arbitrarily large systolic ratio. This shows that the many existing systolic inequalities in Finsler and Riemannian geometry are not purely contact-topological phenomena.
\end{abstract}

\section*{Introduction and main result}

Given a Riemannian metric $g$ on a closed surface $S$, we denote the length of a shortest non-constant closed geodesic by $\ell_{\min}(S,g)$ and consider the scaling invariant ratio
\[
\rho(S,g) := \frac{\ell_{\min}(S,g)^2}{\mathrm{area}(S,g)},
\]
where $\mathrm{area}(S,g)$ denotes the Riemannian area. It is a classical question in systolic geometry to find upper bounds, and possibly optimal ones, for this ratio. When $S$ is not the 2-sphere, the fact that $S$ is not simply connected allows one to bound $\rho(S,g)$ from above by the ratio
\[
\rho_{\rm nc}(S,g) :=\frac{\mathrm{sys}_1(S,g)^2}{\mathrm{area}(S^2,g)},
\]
where $\mathrm{sys}_1(S,g)$ denotes the length of a shortest non-contractible curve on $(S,g)$, which is of course a closed geodesic. When $S$ is the 2-torus $\T^2$, a classical result of Loewner from the end of the 1940s states that $\rho_{\rm nc}(\T^2,\cdot)$ is maximised by the flat metric corresponding to the lattice generated by two sides of an equilateral triangle in $\R^2$. The corresponding maximal value of $\rho_{\rm nc}$ is $2/\sqrt{3}$.
Shortly after, Pu \cite{pu52} considered the case of the projective plane $\R\mathbb{P}^2$ and showed that $\rho_{\rm nc}(\R\mathbb{P}^2,\cdot)$ is maximized by the standard constant curvature metric, with maximal value $\pi/2$. In both cases, the metrics maximizing $\rho_{\rm nc}$ have no contractible closed geodesics, so these metrics maximize also the ratio $\rho$.

For an arbitrary closed non-simply connected surface $S$, the ratio $\rho_{\rm nc}(S,\cdot)$ is bounded from above by the non-optimal universal constant $2$, as shown by Gromov in  \cite[Proposition 5.1.B]{gro83}, but the supremum is in general not achieved. Actually, in \cite[Theorem 0.1.A]{gro83} Gromov proved a far reaching generalization of this result by showing that for any essential $n$-dimensional closed Riemannian manifold $(M,g)$, the quotient
\[
\rho_{\rm nc}(M,g) := \frac{\mathrm{sys}_1(M,g)^n}{\mathrm{vol}(M,g)}
\]
has an upper bound which depends only on the dimension $n$. Here, a non-simply connected closed oriented manifold $M$ is said to be essential if its fundamental class is non-zero in the homology of the Eilenberg-MacLane space $K(\pi_1(M),1)$. This condition generalizes asphericity, that is the fact that all homotopy groups of degree larger than one vanish, and characterizes manifolds for which the ratio $\rho_{\rm nc}$ is bounded from above, see \cite{bab93}.

Going back to surfaces, from the fact mentioned above we deduce that the number $2$ is an upper bound also for $\rho(S,\cdot)$, when $S\neq S^2$. In the case of the two-sphere $S^2$, the first upper bound for $\rho(S^2,\cdot)$ was found by Croke in \cite{cro88} and later improved by several authors, the best one so far being the bound 32 found by Rotman in \cite{rot06}. 

We conclude that $\rho(S,\cdot)$ is always bounded from above when $S$ is a closed surface, and there is an upper bound which is independent of $S$. The ratios $\rho$ and $\rho_{\rm nc}$ generalize to Finsler metrics, by replacing the Riemannian area by the Holmes-Thompson area. Recent results of \`Alvarez Paiva, Balacheff and Tzanev allow to extend the Riemannian bounds to the Finsler setting: The value of the supremum might get larger (and sometimes it does get larger, as in the case of $\T^2$, for which the Loewner metric does not maximize $\rho$ and $\rho_{\rm nc}$, even locally, among Finsler metrics) but is in any case finite. See \cite[Theorems V and IV]{apbt16}.

The next natural generalization of the ratio $\rho$ is to the setting of contact geometry. We recall that a contact form $\alpha$ on a 3-manifold $M$ is a smooth 1-form such that $\alpha\wedge d\alpha$ is a volume form. The kernel of $\alpha$ is a plane distribution and is called the contact structure defined by $\alpha$. In this paper, contact structures are always assumed to be co-orientable, that is, defined by a global contact form. If two contact forms $\alpha_1,\alpha_2$ on $M$ define the same contact structure then the volume forms $\alpha_1\wedge d\alpha_1$ and $\alpha_2\wedge d\alpha_2$ define the same orientation. This orientation is said to be induced by the contact structure, and in this paper all contact $3$-manifolds are oriented by the contact structure. A contact form $\alpha$ induces a nowhere vanishing vector field $R_{\alpha}$ on $M$, which is defined by the conditions
\[
\imath_{R_{\alpha}} d\alpha = 0 , \qquad \alpha(R_{\alpha})=1,
\]
and is called the Reeb vector field for $\alpha$.  Any Reeb vector field on a closed 3-manifold $M$ admits periodic orbits, as proved by Taubes in \cite{tau07}, and we denote by $T_{\min}(M,\alpha)$ the minimum among the periods of all periodic orbits of $R_{\alpha}$. Then it is natural to define
\[
\rho(M,\alpha):= \frac{T_{\min}(M,\alpha)^2}{\mathrm{vol}(M,\alpha\wedge d\alpha)}.
\]
We will refer to this quantity as the systolic ratio of $(M,\alpha)$. It is scaling invariant, because if we multiply $\alpha$ by a non-zero real number $c$ then both $T_{\min}^2$ and the volume get multiplied by $c^2$. The inverse of this quantity appears in \cite{apb14} under the name of systolic volume of $(M,\alpha)$. 

Let $F$ be a smooth Finsler metric on the closed surface $S$. The unit cotangent sphere bundle $S^*_FS$ is the space of cotangent vectors having norm 1 with respect to the dual Finsler metric $F^*$. The canonical Liouville 1-form $p\, dq$ of $T^* S$ restricts to a contact form $\alpha_F$ on $S_F^*S$. Two different Finsler metrics $F$ and $F'$ induce contact forms $\alpha_F$ and $\alpha_{F'}$ which correspond to the same contact structure, once $S_F^*S$ and $S_{F'}^*S$ are identified by means of the radial projection.
The flow of the corresponding Reeb vector field $R_{\alpha_F}$ is precisely the geodesic flow of $F$, once this is read on the cotangent bundle by the Legendre transform. Therefore, $\ell_{\min}(S,F)=T_{\min}(\alpha_F)$. Moreover, the volume of $S^*_FS$ with respect to $\alpha_F \wedge d\alpha_F$ is, essentially by definition, $2\pi$ times the Holmes-Thompson area of $(S,F)$. We conclude that
\[
\rho(S^*_FS,\alpha_F) = \frac{\rho(S,F)}{2\pi},
\]
and the systolic ratio of a contact form is a genuine generalization of the corresponding notion from Riemannian and Finsler geometry. All of this generalizes to higher dimensions, but here we restrict our attention to three-dimensional contact manifolds. 

Seeing Riemannian and Finsler geometry in the larger contact setting is often fruitful. On the one hand, results about existence and multiplicity of Riemannian and Finsler closed geodesics have often natural generalizations to Reeb flows, see e.g.\ \cite{hms15}, and the same holds for some statements about the topological entropy of geodesic flows, see e.g.\ \cite{ms11,fls15,alv16}. On the other hand, techniques from contact geometry have been recently found to be useful to address systolic questions in Riemannian and Finlser geometry, such as the local systolic maximality of Zoll metrics, see \cite{apb14,abhs17,abhs17b}. 

It is therefore a natural question to ask whether the systolic ratio of contact forms inducing a given contact structure on a closed three-manifold also has uniform upper bounds. The purpose of this paper is to give a negative answer to this question. More precisely, we shall prove the following:

\begin{theorem*}\label{main}
Let $\xi$ be a contact structure on a closed $3$-manifold $M$. For every $c>0$ there exists a contact form $\alpha$ satisfying $\ker\alpha=\xi$ and 
\[
\rho(M,\alpha)\geq c.
\]
\end{theorem*}

When applied to the cotangent sphere bundle of an arbitrary closed surface $S$, the above theorem has the following consequence: If $T$ is any positive number, then there exists a fiberwise starshaped domain $A\subset T^*S$ (i.e.\ a connected open neighborhood of the zero-section with smooth boundary $\partial A$ such that for every $q\in S$ the set $\partial A \cap T_q^* S$ is a closed curve in $T^*_q S\setminus \{0\}$ which is transverse to the radial direction) with volume 1 (with respect to the standard symplectic form $dp\wedge dq$ of $T^*S$) and such that every closed orbit of the Reeb vector field on $\partial A$ determined by the restriction of the Liouville form $p\, dq$ has period larger than $T$. When $A$ is fiberwise convex then $\partial A$ is the unit sphere bundle of some Finsler metric and the results of \`Alvarez Paiva, Balacheff and Tzanev mentioned above prevent this phenomenon. This shows that convexity plays an essential role in systolic inequalities.

In the special case of the tight three-sphere $(S^3,\xi_{\rm std})$, that is the standard unit sphere in $\R^4$ endowed with coordinaten $(x_1,y_1,x_2,y_2)$ and with the contact structure $\xi_{\rm std}$ which is given by the kernel of the contact form
\[
\alpha_0 = \frac{1}{2} \sum_{j=1}^2 (x_j \, dy_j - y_j\, dx_j) \Big|_{S^3}
\]
the unboundedness of the systolic ratio was proved in~\cite{abhs17b}. The argument in \cite{abhs17b} starts with the construction of a special symplectomorphism of the disk with a good lower bound on actions of periodic points and with a suitable negative value of the Calabi invariant (see Section \ref{plug} and references therein for the definition of these notions). This disk-map is then embedded as a return map to a disk-like global surface of section of some Reeb flow on $(S^3,\xi_{\rm std})$ in a way that there is a precise ``dictionary'': Actions of periodic points correspond to periods of closed Reeb orbits and the Calabi invariant corresponds to the contact volume, up to explicit additive constants.

The tight three-sphere is very simple from a contact-topological point of view, since it admits the trivial supporting open book decomposition with unknotted binding, disk-like pages and trivial monodromy (see Section \ref{special} below and references therein for the definition of these notions). Moreover, the construction of the disk-map in~\cite{abhs17b} uses the special symmetries of the disk. In order to deal with a general contact structure $\xi$ on a general orientable closed three-manifold $M$ we argue in the following way.   First we use the fact that by a theorem of Giroux \cite{gir02} $\xi$ is supported by an open book decomposition of $M$ to construct a contact form $\alpha$ which is adapted to the open book decomposition and is such that the supports of the monodromy and return maps are contained in a region of small contact area, as visualized in Figure~1. The construction of this special contact form is explained in Section \ref{special}, where the necessary background about open book decompositions is also recalled. The next step is to construct suitable contact forms on a solid torus which are standard near the boundary, have a small contact volume and no closed Reeb orbits with small period. These are constructed in Section \ref{plug}, by building on results from \cite{abhs17b}. Finally, the latter solid tori are used as plugs to modify the contact form $\alpha$ constructed in Section \ref{special}. Indeed, these plugs can be inserted in the region spanned by the large portions of the pages on which the monodromy and first return map are the identity: Their effect is to eat up volume without creating short periodic orbits. In this way we can reduce the contact volume as much as we wish, while keeping $T_{\min}$ bounded away from zero, and hence making the systolic ratio arbitrarily large. This final step is performed in Section \ref{theproof}.

\paragraph{\bf Acknowledgments.} 
The research of A.\ Abbondandolo and B.\ Bramham is supported by the SFB/TRR 191 ``Symplectic Structures in Geometry, Algebra and Dynamics'', funded by the Deutsche Forschungsgemeinschaft. U. Hryniewicz thanks the Floer Center for Geometry (Bochum) for its warm hospitality, and acknowledges the generous support of the Alexander von Humboldt Foundation. U. Hryniewicz is also supported by CNPq grant 309966/2016-7. P. Salom\~ao is supported by FAPESP grant 2016/25053-8 and CNPq grant 306106/2016-7.

\begin{figure}\label{fig1}
\begin{center}
\includegraphics[width=90mm]{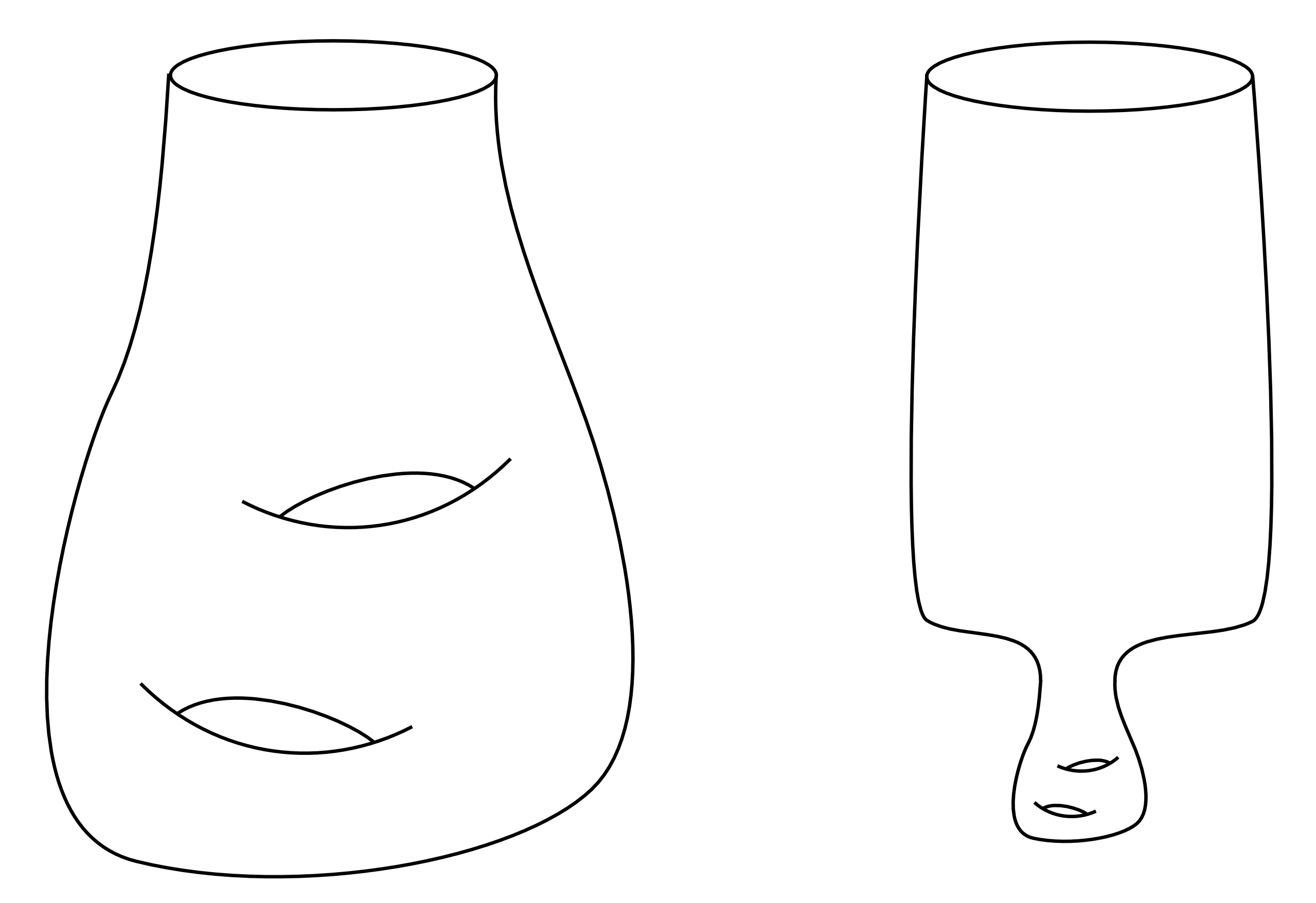}
\caption{\small{On the left a typical page as one would picture it in an open book decomposition. On the right, the same page on an open book supporting a contact structure where the topology, and the supports of the monodromy and return maps are squeezed in a region with small contact area. This leaves lots of space to embed solid tori with high systolic ratio.}}
\end{center}
\end{figure}

\section{A special contact form}
\label{special}

We recall that a global surface of section for a smooth flow $\phi^t$ on a 3-dimensional manifold $M$ is a compact surface $S$ smoothly embedded in $M$ whose boundary $\partial S$ is $\phi^t$-invariant and such that all orbits which are not contained in $\partial S$ meet the surface transversally infinitely many times in the future and in the past. The corresponding first return time function is the smooth function
\[
\tau: S \setminus \partial S \rightarrow (0,+\infty), \qquad \tau(p) := \inf \{ t>0 \mid \phi^t(p)\in S\},
\] 
and the corresponding first return map is the smooth diffeomorphism
\[
\varphi: S \setminus \partial S \rightarrow S \setminus \partial S, \qquad \varphi(p) := \phi^{\tau(p)}(p).
\]
When $\phi^t$ is the Reeb flow of a contact form $\alpha$, the transversality of the flow to $S\setminus \partial S$ implies that the restriction of $d\alpha$ to $S\setminus \partial S$ is an area form. The exactness of $d\alpha$ implies that in the Reeb case the boundary of $S$ is necessarily non-empty. Since Reeb vector fields are non-singular, the finitely many circles forming the boundary of $S$ are periodic orbits. Moreover, it is well known that the contact volume of $M$ coincides with the integral of $\tau$ on $S$, once this surface is  equipped with the area form given by the restriction of $\alpha$:
\begin{equation}
\label{vol-tau}
\mathrm{vol}\,(M,\alpha\wedge d\alpha) = \int_S \tau\, d\alpha.
\end{equation}
See e.g.\ \cite[Lemma 3.5]{abhs17b} for the easy proof.

\begin{proposition}\label{prop_nice_contact_form}
Let $M$ be a closed connected 3-manifold with a contact structure $\xi$. Then there is an embedded compact surface $S\subset M$ with the following property: For every $\epsilon>0$ there exists a contact form $\alpha$ on $M$ satisfying $\xi=\ker \alpha$ such that $S$ is a global surface of section for the Reeb flow of $\alpha$  and:
\begin{enumerate}[(i)]
\item If we orient $S$ by $d\alpha$ and give $\partial S$ the boundary orientation, we have that 
\[
\int_C \alpha = 1 
\]
for every connected component $C$ of $\partial S$. In particular, 
\[
 \int_S d\alpha = \int_{\partial S} \alpha = \ell,
\]
where $\ell\geq 1$ denotes the number of circles forming the boundary of $S$.

\item The first return time function $\tau$ of the Reeb flow of $\alpha$ extends to a smooth function on $S$ and the corresponding first return map $\varphi$ extends to a smooth diffeomorphism of $S$ onto itself.

\item There exists an open tubular neighborhood $U$ of $\partial S$ in $S$ such that the support of $\varphi$ is contained in $S\setminus U$ and
\[
\int_{S\setminus U} d\alpha < \epsilon.
\]

\item $\|\tau-1\|_{\infty} < \epsilon$ and $\tau$ is constantly equal to $1$ on the neighborhood $U$ from (iii).
\end{enumerate}
\end{proposition}

In order to prove this proposition, we need to recall some facts about open book decompositions. See \cite[Section 2]{etn06} for more details. An open book decomposition of an oriented closed $3$-manifold $M$ is a pair $(\Pi,L)$, where $L\subset M$ is a smooth link and 
\[
\Pi:M\setminus L \to \R/2\pi\Z
\]
is a smooth (locally trivial) fibration which near $L$ is a normal angular coordinate. More precisely, there exists a compact neighborhood $N$ of $L$ and a diffeomorphism $N \simeq L \times \D$ such that $L \simeq L \times \{0\}$ and $\Pi|_{N \setminus L}$ is represented as
\[
(\lambda ,r e^{ix})\mapsto x,
\] 
where $\lambda$ belongs to $L$ and $(r,x)\in (0,1]\times \R/2\pi \Z$ are polar coordinates on the punctured closed unit disc $\D\setminus \{0\}\subset \C$. 

The link $L$ is called the binding and the fibers of $\Pi$ are called the pages of the open book decomposition. The closure of each page is a Seifert surface for $L$, that is, a compact smoothly embedded surface with boundary $L$. The co-orientation of the pages given by the map $\Pi$ and the ambient orientation induce an orientation of the pages. The link $L$ inherits the boundary orientation.

Let $S$ be  the closure in $M$ of the page $\Pi^{-1}(0)$.  The coordinates introduced above induce coordinates $(\lambda,r) \in \partial S\times [0,1]$ on the compact neighborhood $N\cap S$ of $\partial S$ in $S$, where $[0,1]$ is seen as the intersection of $\D$ with the positive real half-axis. 

The vector field $\partial_x$ on $N \simeq L \times \D$ is smooth, vanishes at $L\simeq L \times \{0\}$ and generates the group of rotations
\begin{equation}
\label{rotations}
\bigl(t,(\lambda,\rho e^{ix}) \bigr) \mapsto (\lambda,\rho e^{i(x+t)}).
\end{equation}
Let $Z$ be a smooth vector field on $M$ which coincides with $\partial_x$ on $N$ and is positively transverse to all the pages, meaning that $d\Pi \circ Z >0$ on $M\setminus L$. The set of such vector fields is obviously convex and in particular path connected.

The surface $S$ is a global surface of section for the flow of $Z$. The form (\ref{rotations}) of the flow of $Z$ on $N$ implies that the corresponding first return map is the identity on the set $N \cap (S\setminus \partial S)$, and hence it extends smoothly to the boundary and gives us an orientation preserving diffeomorphism 
\[
h: S \rightarrow S,
\]
which is the identity on $N \cap S$. The diffeomorphism $h$ is called the monodromy map of the open book decomposition which depends on the choice of the vector field $Z$. However as the set of such vector fields is path connected the isotopy class of $h$ within the space of diffeomorphisms of $S$ that are the identity in a neighborhood of $\partial S$ is uniquely determined.   

The mapping torus of $h$ is the smooth 3-manifold with boundary 
\begin{equation}
\label{mapping-torus}
S(h) := \left([0,2\pi]\times S\right)/ \sim \qquad (2\pi,q) \sim (0,h(q)).
\end{equation}
It is equipped with the fibration 
\begin{equation}
\label{fibration}
x: S(h) \to \R/2\pi\Z,  
\end{equation}
given by the projection onto the first component. The piece $([0,2\pi]\times (S\cap N))/\sim$ is diffeomorphic to $\R/2\pi\Z \times (S\cap N)$, because $h$ is the identity on $S\cap N$, and is equipped with coordinates 
\begin{equation}
\label{coordinate}
(x,\lambda,r) \in \R/2\pi\Z \times \partial S \times[0,1].
\end{equation}
In particular, the boundary of $S(h)$ is the product $\R/2\pi \Z \times \partial S$, which is a union of finitely many 2-tori.

The flow of $Z$ preserves the family of leaves in the neighborhood $N$ and is transverse to the interiors of the leaves.   Thus we may renormalise $Z$ outside of $N$ so that its flow preserves the leaves also in $M\backslash N$.    Then the flow of $Z$ induces a diffeomorphism  between the interior of $S(h)$ and $M\setminus L$, which identifies the fibrations $x$ and $\Pi$. Moreover, $M$ is obtained from $S(h)$ by collapsing each boundary torus onto a circle. More precisely, we can identify $M$ with the quotient

\begin{equation}
\label{Macca}
\bigl(S(h) \sqcup (\partial S\times\D)\bigr)/\sim
\end{equation}
where $(x,\lambda,r) \sim (\lambda,r e^{ix})$. Here, the smooth structure near the link is induced by the identification $\partial S \times \D \simeq N$.

\begin{proof}[Proof of Proposition \ref{prop_nice_contact_form}]
By a theorem of Giroux \cite[Theorem 3]{gir02} (see also \cite[Theorem 4.6]{etn06} for a detailed proof), the contact structure $\xi$ is supported by some open book decomposition $(\Pi,L)$ of $M$: This means that $\xi=\ker \alpha$, where $\alpha$ is a contact form such that:
\begin{enumerate}[({A}1)]
\item the restriction of $d\alpha$ to each page is a positive area form;
\item $\alpha$ is positive on $L$.
\end{enumerate}

As smoothly embedded compact surface $S$ we choose the closure of the page $\Pi^{-1}(0)$. We freely draw from the notation about open book decompositions established above; in particular, $N$ is the closed tubular neighborhood of $L=\partial S$ with diffeomorphism $N\simeq L\times \D$ and $h:S \rightarrow S$ is a choice of monodromy map which is supported in $S\setminus N$.

If $\alpha$ and $\alpha'$ are contact forms satisfying (A1) and (A2), then $\ker \alpha$ and $\ker \alpha'$ are isotopic (see \cite[Proposition 2]{gir02} and \cite[Proposition 3.18]{etn06}). In this case, Gray's stability theorem (see \cite[Theorem 2.2.2]{gei08}) implies the existence of a diffeomorphism bringing $\ker \alpha$ into $\ker \alpha'$. Since the properties (i), (ii), (iii) and (iv) that we wish our contact form to have are preserved by the action of a diffeomorphism, it is enough to find a contact form $\alpha$ which satisfies them together with the conditions (A1) and (A2) above, i.e. without explicitly assuming that $\ker\alpha$ is the given contact structure.

Denote by $C_1,\dots,C_{\ell}$ the connected components of $\partial S$. The tubular neighborhood $N\cap S$ of $\partial S$ in $S$ is the union of $\ell$ closed annuli $A_j$ carrying coordinates $(\theta,r)\in \R/2\pi \Z \times [0,1]$, where $\theta$ is an orientation preserving parametrization of $C_j$. Here the boundary component $C_j$ corresponds to $r=0$.
The fact that $C_j$ is given the boundary orientation from $S$ implies that $2r\,d\theta\wedge dr$ is a positive 2-form on each half-open annulus $A_j\setminus \partial S$. Choose a 2-form $\Omega$ on $S$ such that $\Omega>0$ on $S\setminus L$, 
\[
\int_S\Omega = 2\pi \ell
\]
and 
\[
\Omega = 2r\, d\theta \wedge dr \qquad \mbox{on each closed annulus } A_j':= \R/2\pi \Z \times [0,\rho]\subset A_j,
\]
for a suitably small number $\rho\in (0,1]$.

Perhaps after shrinking $\rho$, we can assume that there exists a primitive $\eta$ of $\Omega$ agreeing with $(1-r^2)d\theta$ on each $A_j'$. To see this, choose any $1$-form $\sigma_0$ on $S$ which agrees with $(1-r^2)d\theta$ near the $C_j$'s. Then $\Omega-d\sigma_0$ is a $2$-form with support contained in $S\setminus\partial S$. Since
\[
\int_S \bigl( \Omega - d\sigma_0 \bigr) = 2\pi \ell - \int_{\partial S} \sigma_0 = 2\pi \ell - 2\pi \ell = 0,
\]
we can invoke de Rham's theorem to find a $1$-form $\sigma_1$ supported in $S\setminus\partial S$ and such that $\Omega-d\sigma_0=d\sigma_1$. Then $\eta = \sigma_0+\sigma_1$ has the desired property. 

Let $\beta:\R\to[0,1]$ be a smooth function satisfying $\beta(0)=0$, $\beta(2\pi)=1$ and $\supp(\beta')\subset (0,2\pi)$. The family of $1$-forms 
\begin{equation}\label{form_alpha_tilde_epsilon}
\tilde\alpha_s := dx + s \bigl( (1-\beta(x))  \eta + \beta(x) h^*\eta \bigr), \qquad s\in (0,+\infty),
\end{equation}
is well defined on the mapping torus $S(h)$ of $h$ which is defined in (\ref{mapping-torus}). Here $x$ denotes the fibration (\ref{fibration}), which corresponds to the fibration $\Pi$ under the identification of the interior of $S(h)$ with $M\setminus L$.

From now on, we shall identify $M$ with the manifold 
\[
\left( S(h) \sqcup \bigsqcup_{j=1}^{\ell} \bigl( \R/2\pi\Z \times \rho\D \bigr) \right) / \sim
\]
as in (\ref{Macca}), where a point in $\R/2\pi \Z\times A_j'\subset S(h)$ with coordinates $(x,\theta,r)$, is identified with the point $(\theta,r e^{ix})$ in the corresponding copy of $\R/2\pi\Z \times \rho \D$. With this identification, the set 
\[
V := \bigsqcup_{j=1}^{\ell} \bigl( \R/2\pi\Z \times \rho \D \bigr)
\]
is a compact tubular neighborhood of $L$ in $M$.

We will think of $\tilde{\alpha}_s$ as defined in $M\setminus L \simeq S(h)\setminus \partial S(h)$. The fact that $h$ is the identity on $A_j$ and the form of $\eta$ on $A_j'$ imply that
\begin{equation}
\label{alphaonV}
\tilde{\alpha}_s=dx + s (1-r^2)\, d\theta \qquad \mbox{on }V\setminus L,
\end{equation}
with respect to the standard coordinates $(\theta,r,x)\in \R/2\pi \Z \times (0,\rho] \times \R/2\pi \Z$ on $V \setminus L$.

Fix $\delta>0$. There exists $s_1>0$, depending only on $\rho$ and $\delta$, with the following properties: If $s\in(0,s_1)$ then there are smooth functions $f,g:[0,\rho] \rightarrow [0,+\infty)$ defining a curve in the complex plane
\[
\gamma: [0,\rho] \rightarrow \C, \qquad \gamma(r) = f(r) + i g(r),
\]
satisfying:
\begin{enumerate}[({B}1)] 
\item $\gamma(r) = 1 + is(1-r^2)$ on $[r_1,\rho]$, for some $r_1 \in (0,\rho)$; 
\item $g'<0$ on $(0,\rho]$.
\item $\gamma([0,r_0]) \subset \{x+iy \mid x+y=1+\delta\}$ for some $r_0\in (0,r_1)$ such that
\[
g(r_0)-g(\rho)=g(r_0)-s(1-\rho^2) \leq 2\delta.
\]
Moreover, $\gamma(r) = r^2 + i (1+\delta-r^2)$ when $r$ is close enough to $0$.
\item The derivative of the argument of $\gamma$, that is the function $$ \frac{g'f-f'g}{f^2+g^2} $$ is negative on $(0,\rho]$.
\item The derivative of the argument of $\gamma'$, that is the function $$ \frac{g''f'-f''g'}{(f')^2+(g')^2} $$ is non-positive on $[0,\rho]$.
\end{enumerate}

\begin{figure}\label{fig2}
\begin{center}
\includegraphics[width=70mm]{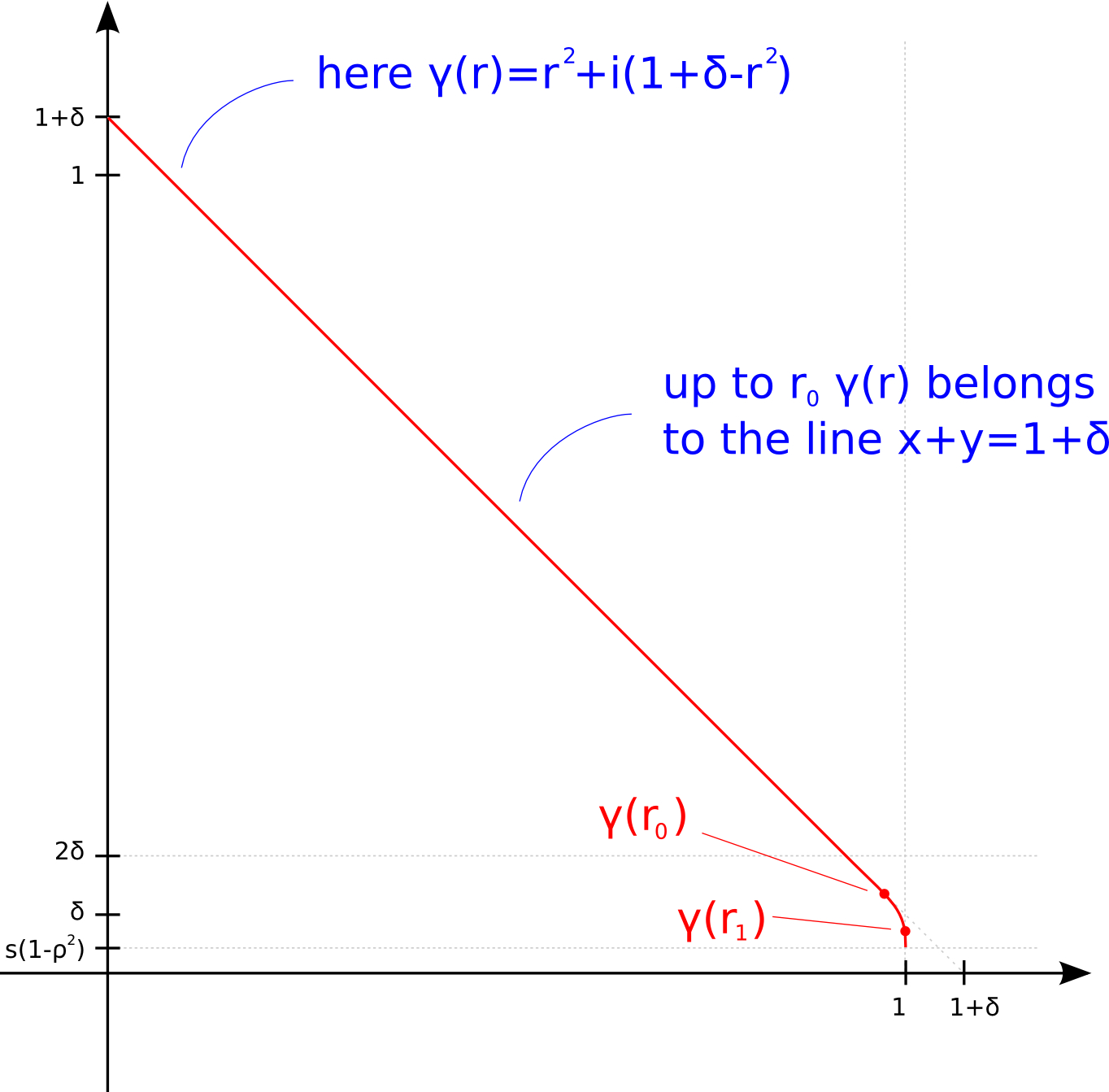}
\caption{\small{Functions $f,g$ with the above properties exist.}}
\end{center}
\end{figure}

Finally, we define a smooth $1$-form $\alpha_s$ on $M$ by 
\[
\alpha_s = \left\{ \begin{array}{ll} \displaystyle{\frac{\tilde{\alpha}_s}{2\pi(1+\delta)}} & \mbox{on } M\setminus V, \\ \\
\displaystyle{\frac{f(r)\,dx + g(r)\,d\theta}{2\pi(1+\delta)}} & \mbox{on } V. \end{array} \right.
\]
The smoothness of $\alpha_{s}$ near the boundary of $V$ follows from (\ref{alphaonV}) and (B1).
By (B3) we have the formula
\begin{equation}
\label{alphaonL}
\alpha_{s} = \frac{r^2 \, dx + (1+\delta-r^2)\, d\theta}{2\pi(1+\delta)} \qquad \mbox{near $L$},  
\end{equation}
which shows that $\alpha_{s}$ is smooth in a neighborhood of $L$. Therefore, $\alpha_{s}$ is a smooth 1-form on $M$ for every $s\in (0,s_1)$.

We claim that there exists $s_2 \in (0,s_1)$, depending on $\delta,\rho,h,\eta,\beta$, such that if $s \in (0,s_2)$ then $\alpha_s$ is a contact form and satisfies conditions (A1) and (A2). Let us prove this claim. By (B3) the 1-form $\alpha_{s}$ has the following expression on $L$:
\begin{equation}
\label{alphaonL}
\alpha_{s} = \frac{1}{2\pi} \, d\theta \qquad \mbox{on } L \simeq \R/2\pi \Z \times \{0\} \subset \R/2\pi \Z \times \rho \D ,
\end{equation}
which shows that condition (A2) holds for every $s\in (0,s_1)$. On $V$ we have
\begin{equation}
\label{dalphaonV}
d\alpha_s = \frac{f'dr \wedge dx-g'd\theta\wedge dr}{2\pi(1+\delta)}.
\end{equation}
The restriction of $d\alpha_{s}$ on the page $\Pi^{-1}(x_0) = x^{-1}(x_0)$ intersected with $V$ is
\[
-\frac{g'd\theta\wedge dr}{2\pi (1+\delta)},
\]
which is a positive area form thanks to (B2) and the positivity of $d\theta\wedge dr$. So the condition (A1) holds on $V$ for every $s\in (0,s_1)$. Moreover,
\[
\alpha_s \wedge d\alpha_s = \frac{f'g-fg'}{4\pi^2 (1+\delta)^2} \ dx\wedge d\theta\wedge dr = \frac{f'g-fg'}{4\pi^2 (1+\delta)^2 r} \ d\theta\wedge (r\, dr \wedge dx).
\]
By the last part of (B3),
\[
\frac{f'g-fg'}{r} =2(1+\delta)
\]
when $r$ is close enough to $0$. Together with (B4), this implies that the smooth function $(f'g-fg')/r$ is strictly positive on $V$. This implies that $\alpha_{s}$ is a contact form on $V$ for every $s\in (0,s_1)$.

Now we analyze $\alpha_s$ on $M\setminus V$. The formula
\[
\alpha_s=\frac{dx+s \bigl((1-\beta)\eta+\beta h^*\eta \bigr)}{2\pi(1+\delta)} \qquad \mbox{on }M\setminus V
\]
gives
\begin{equation}
\label{dalpha}
d\alpha_s = s \, \frac{\beta'dx\wedge (h^*\eta-\eta) +  \omega}{2\pi(1+\delta)} \qquad \mbox{on }M\setminus V,
\end{equation}
where $\omega$ is the smooth 2-form on $M\setminus L$ with kernel $\partial_x$ and whose restriction to the page $\Pi^{-1}(x) \simeq S\setminus \partial S$ is the positive area form
\[
\omega_x := (1-\beta(x))d\eta+\beta(x) h^*d\eta.
\]
The restriction of $d\alpha_{s}$ to the page $\Pi^{-1}(x_0)=x^{-1}(x_0)\simeq S\setminus S$ intersected with $M\setminus V$ is a positive multiple of the positive area form $\omega_{x_0}$, so $\alpha_{s}$ satisfies (A1) on $M\setminus V$ for every $s \in (0,s_1)$. Moreover,
\[
\frac{4\pi^2 (1+\delta)^2}{s} \, \alpha_s \wedge d\alpha_s = dx\wedge \omega + s \ \beta' \eta \wedge dx \wedge h^*\eta  \qquad \mbox{on } M\setminus V.
\]
Since $dx\wedge \omega$ is a volume form on $M\setminus V$, the above formula shows that we can choose $s_2\in(0,s_1]$, depending on the data $h,\eta,\beta$, such that $\alpha_s$ is a contact form on $M\setminus V$ whenever $s\in (0,s_2)$. We conclude that $\alpha_s$ is a contact form satisfying conditions (A1) and (A2), for every $s\in(0,s_2)$. 

The next task is to understand the Reeb flow of $\alpha_s$ and to show it fullfills the requirements (i), (ii), (iii) and (iv) with respect to the surface $S$, when $s$ is small enough. Condition (i) is actually fulfilled for any $s\in (0,s_1)$, due to (\ref{alphaonL}), so we need to focus only on (ii), (iii) and (iv).

Denote by $R_{s}$ the Reeb vector field of $\alpha_{s}$. We start by analysing the flow of $R_{s}$ on $M\setminus V$. Since $\beta'$ is supported in $(0,2\pi)$, $h^* \eta-\eta$ is compactly supported in $S\setminus V$ and $\omega$ restricts to an area form on each page, there is a unique smooth vector field $Y$ on $M$ which is supported in $M\setminus V$, is tangent to the pages and satisfies
\[
\omega(Y,\cdot) = -\beta'(x) (h^*\eta-\eta) \qquad \mbox{on } \Pi^{-1}(x).
\]
By (\ref{dalpha}), we have the identity
\[
\begin{split}
\frac{2\pi(1+\delta)}{s} \ d\alpha_s(\partial_x+Y,\cdot) &= \beta'(h^*\eta-\eta) - \beta'(h^*\eta-\eta)(Y)\, dx + \omega(Y,\cdot) \\ &=\beta'(h^*\eta-\eta) + \omega(Y,Y)\, dx - \beta'(h^*\eta-\eta) = 0,
\end{split}
\]
which shows that the non-vanishing vector field $\partial_x+Y$ is in the kernel of $d\alpha_{s}$. Therefore, the Reeb vector field $R_{s}$ of $\alpha_{s}$  has the form
\begin{equation}\label{form_Reeb_vector_outofV}
R_s = \frac{\partial_x+Y}{\alpha_s(\partial_x+Y)} \qquad \mbox{on } M\setminus V,
\end{equation}
for every $s\in (0,s_2)$. The fact that $Y$ is supported in $M\setminus V$ and the form of $V$ imply that $M\setminus V$ is invariant for the Reeb flow, and hence the same is true for its complement $V$. Moreover, the above formula implies that $R_{s}$ is transverse to the portion of the pages $x^{-1}(x_0)$ lying in $M\setminus V$, and in particular to $S\setminus V$.

Next we study the Reeb flow on $V$. By the form of $\alpha_{s}$ on $V$, $R_{s}$  has the form
\begin{equation}
\label{ReebonV}
R_{s} = 2\pi(1+\delta) \frac{f' \partial_{\theta} - g' \partial_x}{f'g-g'f} \qquad \mbox{on } V.
\end{equation}
The condition (B2) implies that $R_{s}$ is transverse the the portions of the pages $x^{-1}(x_0)$ which lie in $V$, and in particular to $(S\setminus \partial S)\cap V$. 
The expression for $R_{s}$ near $L$ becomes, thanks to (B3), 
\[
R_{s} = 2\pi(\partial_{\theta} + \partial_x) \qquad \mbox{for } (\theta,r,x)\in \R/2\pi \Z \times [0,r_0] \times \R/2\pi \Z.
\]
In particular, since $\partial_x$ vanishes on $L$, 
\[
R_{s} = 2\pi \,\partial_{\theta}\qquad \mbox{on } L,
\]
and the link $L$ consists of periodic orbits of $R_{s}$ of period $1$. 

We conclude that the vector field $R_{s}$ is transverse to the pages $x^{-1}(x_0)$ and leaves the binding $L$ invariant. In particular, in view of the form of $R_s$ near $L$ described above, $S$ is a global surface of section for the Reeb flow, and the first return time function and first return map
\[
\tau: S \setminus \partial S \rightarrow (0,+\infty), \qquad \varphi: S \setminus \partial S \rightarrow S \setminus \partial S
\]
are well defined. By (\ref{ReebonV}), on $(S\setminus \partial S)\cap V$ the first return time function is explicitly given by the formula
\begin{equation}
\label{tau}
\tau = \frac{g'f-f'g}{(1+\delta)g'} \qquad \mbox{on } (S\setminus \partial S)\cap V,
\end{equation}
and the first return time map by
\begin{equation}
\varphi: (r,\theta) \mapsto \left( r,\theta-2\pi \frac{f'(r)}{g'(r)} \right) \qquad \mbox{on } (S\setminus \partial S) \cap V.
\end{equation}
By (B3), we have  
\begin{equation}
\label{vicinobordo}
\tau(r,\theta) = 1 \qquad \mbox{and} \qquad \varphi(r,\theta) = (r,\theta-2\pi) \qquad \forall (r,\theta)\in (0,r_0]\times \R/2\pi \Z.
\end{equation}
This implies that $\tau$ and $\varphi$ extend smoothly to the boundary of $S$ and proves (ii). 

By (\ref{tau}), the restriction of the first return time function $\tau$ to $S\cap V$ depends only on $r$ and we have
\[
\partial_r \tau =  \frac{g(f'g''-g'f'')}{(1+\delta)(g')^2}.
\]
By (B5), this function is non-positive on $[0,\rho]$. Together with (\ref{vicinobordo}) and
\[
\tau(\rho,\theta) = \frac{1}{1+\delta},
\]
where we have used (\ref{tau}) and (B1), we find the bounds
\[
\frac{1}{1+\delta} \leq \tau \leq 1 \qquad \mbox{on } S\cap V,
\]
from which
\begin{equation}
\label{bdontau}
\sup_{S\cap V} |\tau-1| \leq 1 - \frac{1}{1+\delta} = \frac{\delta}{1+\delta} < \delta.
\end{equation}
From (\ref{form_Reeb_vector_outofV}) we have
\begin{equation}
\label{diffo}
dx(R_{s}) = \frac{1}{\alpha_{s}(\partial_x + Y)} \qquad \mbox{on } M\setminus V.
\end{equation}
Since the function 
\[
\alpha_{s}(\partial_x + Y) = \frac{1}{2\pi (1+\delta)} \tilde{\alpha}_{s}(\partial_x + Y) = \frac{1}{2\pi (1+\delta)} \bigl( 1 + s ((1-\beta) \eta(Y) + \beta h^* \eta(Y))\bigr) 
\]
converges to $1/(2\pi(1+\delta))$ for $s\rightarrow 0$ uniformly on $M\setminus V$, by (\ref{diffo}) the function $dx(R_{s})$ converges to $2\pi(1+\delta)$ for $s\rightarrow 0$ uniformly on $M\setminus V$. This implies that $\tau$ converges uniformly to $1/(1+\delta)$ on $S\setminus V$. Together with (\ref{bdontau}), this implies that there exists $s_3\in (0,s_2]$ such that for every $s\in (0,s_3)$ we have
\begin{equation}
\label{bdontau2}
\|\tau-1\|_{\infty} < 2\delta.
\end{equation}
Define $U$ to be the open tubular neighborhood of $L=\partial S$ in $S$ consisting of those $(r,\theta)$ in $S\cap V$ with $0\leq r < r_0$. By (\ref{vicinobordo}), $\tau=1$ and $\varphi=\mathrm{id}$ on $U$. Together with (\ref{bdontau2}), this proves that (iv) holds, by choosing $\delta\leq \epsilon/2$.

The identity (\ref{dalpha}) gives us a constant $C$, depending only on $\beta$, $\eta$ and $h$, such that
\[
\int_{S\setminus V} d\alpha_{s} \leq C s \qquad \forall s\in (0,s_1).
\]
We now fix a $s_4\in (0,s_3]$ such that $C s_4<\delta$.
By (\ref{dalphaonV}), the $d\alpha_{s}$-area of the annulus in $A_j'$ corresponding to the values of $r$ in the interval $[r_0,\rho]$ is
\[
\frac{1}{2\pi(1+\delta)} \int_{\R/2\pi \Z \times [r_0,\rho]} (-g') \, d\theta\wedge dr = \frac{1}{1+\delta} \bigl( g(r_0) - g(\rho) \bigr) \leq \frac{2\delta}{1+\delta}< 2\delta,
\]
where the first upper bound follows from (B3). If $s$ is in the interval $(0,s_4)$, the above two inequalities imply that
\[
\int_{S\setminus U} d\alpha_{s} = \int_{S\setminus V} d\alpha_{s} + \frac{1}{2\pi(1+\delta)} \int_{\R/2\pi \Z \times [r_0,\rho]} (-g') \, d\theta\wedge dr < 3\delta,
\]
and the conclusion (iii) follows by choosing $\delta=\epsilon/3$.
\end{proof}

\section{Construction of the plug}
\label{plug}

The aim of this section is to show how some results from \cite{abhs17b} can be used in order to build a contact form on a solid torus such that the contact volume is small and all closed orbits of the corresponding Reeb flow have large period. In the next section, this solid torus will be used as a plug to modify the special contact form which we constructed in the previous section. Here is the statement which summarizes the properties of the plug.

\begin{proposition}
\label{prop0}
Fix positive numbers $r$ and $\epsilon$. Let $\lambda$ be a primitive of the standard area form $dx\wedge dy$ on the closed disc $r\mathbb{D}$. Then there exists a smooth contact form $\beta$ on the solid torus $r\mathbb{D} \times \R/\Z$ with the following properties:
\begin{enumerate}[({a}1)]
\item $\beta = \lambda + ds$ in a neighborhood of $\partial(r \D) \times \R/\Z$, where $s$ denotes the coordinate on $\R/ \Z$; in particular, the Reeb vector field $R_{\beta}$ of $\beta$ coincides with $\partial_s$ near the boundary of $r\D \times \R/\Z$, and its flow is globally well-defined;
\item the contact form $\beta$ is smoothly isotopic to the contact form $\lambda + ds$ on $r\D \times \R/ \Z$ through a path of contact forms which agree with $\lambda + ds$ in a neighborhood of $\partial (r\D) \times \R/ \Z$;
\item all the closed orbits of $R_{\beta}$ have period at least 1;
\item $\mathrm{vol}\,(r\D \times \R/ \Z, \beta\wedge d\beta) < \epsilon$.
\end{enumerate}
\end{proposition}

Before discussing the proof of the above proposition, we recall the definition of the Calabi invariant for compactly supported area-preserving diffeomorphisms of the plane.

Endow the plane $\R^2$ with the standard area form $dx\wedge dy$ and let $\mathrm{Diff}_c(\R^2,dx\wedge dy)$ be the group of compactly supported area-preserving diffeomorphisms of $\R^2$. Let $\varphi$ be an element of $\mathrm{Diff}_c(\R^2,dx\wedge dy)$ and let $\lambda$ be a primitive of $dx\wedge dy$.
To $\varphi$ and $\lambda$ we can associate the unique compactly supported smooth function $\sigma_{\varphi,\lambda}: \R^2 \rightarrow \R$ satisfying
\[
\varphi^* \lambda - \lambda = d\sigma_{\varphi,\lambda}.
\]
This function is called the action of $\varphi$ with respect to $\lambda$. Its value at fixed points of $\varphi$ is independent of the choice of the primitive $\lambda$, and so is its integral
\[
\mathrm{CAL}(\varphi) = \int_{\R^2} \sigma_{\varphi,\lambda} \, dx\wedge dy,
\]
which is called the Calabi invariant of $\varphi$. The Calabi invariant defines a homomorphism from $\mathrm{Diff}_c(\R^2,dx\wedge dy)$ onto $\R$.

Let $\lambda_0$ be the following radially symmetric primitive of $dx\wedge dy$
\[
\lambda_0 := \frac{1}{2} ( x\, dy - y \, dx).
\]
We shall make use of the following result:

\begin{proposition}
\label{prop1}
Fix positive numbers $r$ and $L$. For every $\epsilon>0$ there exists a positive integer $n$ and an area-preserving diffeomorphism
\[
\varphi\in  \mathrm{Diff}_c(\mathrm{int}(r\mathbb{D}),dx\wedge dy)
\]
such that the following properties hold:
\begin{enumerate}[({b}1)]
\item $\sigma_{\varphi,\lambda_0} \geq - L + L/n $;
\item $\mathrm{CAL}(\varphi)< - L\pi r^2 + \epsilon$;
\item all the fixed points of $\varphi$ have non-negative action;
\item all the periodic points of $\varphi$ which are not fixed points have period at least $n$.
\end{enumerate}
\end{proposition}

This result is proved in \cite[Proposition 2.27]{abhs17b} for $r=1$. The general case follows by a simple rescaling argument, using that if $\varphi$ is in $\mathrm{Diff}_c(\R^2,dx\wedge dy)$ then the rescaled map
\[
\varphi_r(z) := r \varphi \left( \frac{z}{r} \right) 
\]
is also in $\mathrm{Diff}_c(\R^2,dx\wedge dy)$ and 
\[
\sigma_{\varphi_r,\lambda_0}(z) = r^2 \sigma_{\varphi,\lambda_0}\left( \frac{z}{r} \right), \qquad  \mathrm{CAL}(\varphi_r)= r^4 \, \mathrm{CAL}(\varphi).
\]

The last ingredient which we need is the following result from \cite{abhs17b}[Propo\-sition 3.1], which allows us to lift an area preserving diffeomorphism of a disc to a Reeb flow on a solid torus. 

\begin{proposition}
\label{prop2}
Fix positive numbers $r$ and $L$.  Let $\varphi\in \mathrm{Diff}_c(\mathrm{int}(r\mathbb{D}),dx\wedge dy)$ and assume that the function
\[
\tau:= \sigma_{\varphi,\lambda_0} + L
\]
is positive on $r\D$. Then there exists a smooth contact form $\beta$ on the solid torus $r\D \times \R/L \Z$ with the following properties:
\begin{enumerate}[({c}1)]
\item $\beta = \lambda_0 + ds$ in a neighborhood of $\partial(r \D) \times \R/L \Z$, where $s$ denotes the coordinate on $\R/L \Z$; in particular, the Reeb vector field $R_{\beta}$ of $\beta$ coincides with $\partial_s$ near the boundary of $r\D \times \R/L \Z$, and its flow is globally well-defined;
\item the surface $r\D \times \{0\}$ is transverse to the flow of $R_{\beta}$, and the orbit of every point in $r\D \times \R/L \Z$ intersects $r\D \times \{0\}$ both in the future and in the past;
\item the first return map and the first return time of the flow of $R_{\beta}$ associated to the surface $r\D\times \{0\} \cong r\D$ are the map $\varphi$ and the function $\tau$;
\item $\mathrm{vol}\,(r\D \times \R/L \Z, \beta\wedge d\beta) = L \,\pi r^2+ \mathrm{CAL}(\varphi)$;
\item the contact form $\beta$ is smoothly isotopic to $\lambda_0 + ds$ on $r\D \times \R/L \Z$ through a path of contact forms which agree with $\lambda_0 + ds$ in a neighborhood of $\partial (r\D) \times \R/L \Z$.
\end{enumerate}
\end{proposition}

Again, Proposition 3.1 in \cite{abhs17b} is stated for $r=1$, but the general case follows by rescaling. In fact consider $r,L$ and $\varphi$ as in the statement of Proposition~\ref{prop2}. Then, as explained above, we have $\sigma_{\varphi_{r^{-1}},\lambda_0}(z) = r^{-2} \sigma_{\varphi,\lambda_0}\left( rz \right)$. Applying~\cite[Proposition~3.1]{abhs17b} to the map $\varphi_{r^{-1}}$ with $Lr^{-2}$ in the place of $L$ we get a contact form $\beta_{r^{-1}}$ on $\D\times \R/Lr^{-2}\Z$ satisfying the desired conclusions. Consider now the diffeomorphism $\Phi_{r^{-1}}:r\D\times\R/L\Z \to \D\times\R/Lr^{-2}\Z$ defined as $(z,s) \mapsto (r^{-1}z,r^{-2}s)$. Direct calculations reveal that $\beta = r^2\Phi_{r^{-1}}^*\beta_{r^{-1}}$ is the desired contact form.

The statement of \cite[Proposition 3.1]{abhs17b} is actually slightly more general, since it allows $\lambda_0$ to be replaced by a more general primitive of $dx\wedge dy$, and gives more properties of the contact form $\beta$.

Building on the above two propositions, it is now easy to prove Proposition \ref{prop0}.

\begin{proof}[Proof of Proposition \ref{prop0}]
The first step is to prove the existence of a contact form $\beta_0$ on $r\D \times \R/\Z$ which satisfies the required conditions (a1)-(a4), but where in (a1) and (a2) the primitive $\lambda$ of $dx\wedge dy$ is the radially symmetric primitive $\lambda_0$.
Let $n\in \N$, $\varphi\in \mathrm{Diff}_c(\mathrm{int}(r\mathbb{D}),dx\wedge dy)$ be given by an application of Proposition \ref{prop1} with $L=1$. By statement (b1) in this proposition, we have the lower bound
\begin{equation}
\label{azionebuona}
1 + \sigma_{\varphi,\lambda_0} \geq \frac{1}{n}.
\end{equation}
In particular, the function
\[
\tau:= 1 + \sigma_{\varphi,\lambda_0}
\]
is everywhere positive. Then Proposition \ref{prop2} implies the existence of a contact form $\beta_0$ on the solid torus $r\mathbb{D} \times \R/\Z$ which satisfies the conditions (c1)-(c5) with $L=1$. We just need to check that $\beta_0$ satisfies (a1)-(a4) with respect to $\lambda_0$. Conditions (a1) and (a2) are precisely (c1) and (c5). Condition (a4) follows from (c4) and (b2):
\[
\mathrm{vol}\,(r\D \times \R/ \Z, \beta_0\wedge d\beta_0) \stackrel{(c4)}{=} \pi r^2 + \mathrm{CAL}(\varphi)  \stackrel{(b2)}{<} \epsilon.
\]

There remains to prove that all closed orbits of the Reeb flow of $\beta_0$ have period at least 1. By (c2) and (c3), closed orbits of $R_{\beta_0}$ are in one-to-one correspondence with periodic points of $\varphi$, and if $z\in r\mathbb{D}$ is a periodic point of $\varphi$ with minimal period $k\in \N$, then the corresponding closed orbit has period
\[
T:= \sum_{j=0}^{k-1} \tau(\varphi^j(z)).
\]
When $k=1$, $z$ is a fixed point of $\varphi$ and by (b3) we have
\[
T=\tau(z) = 1 + \sigma_{\varphi,\lambda_0}(z) \geq 1.
\]
By condition (b4), all periodic points of $\varphi$ which are not fixed points have period $k\geq n$. If $z$ is such a point, condition (b1) gives us
\[
T = \sum_{j=0}^{k-1} \tau(\varphi^j(z)) = \sum_{j=0}^{k-1} \bigl(1+\sigma_{\varphi,\lambda_0}(\varphi^j(z))\bigr) \geq \sum_{j=0}^{k-1} \frac{1}{n} = \frac{k}{n} \geq 1.
\]
This proves (a3) and concludes the first step.

Now we wish to modify $\beta_0$ in a neighborhood of the boundary in order to obtain (a1) and (a2) with respect to the given primitive $\lambda$ of $dx\wedge dy$, while keeping conditions (a3) and (a4). The 1-form $\lambda-\lambda_0$ is closed and hence exact on $r\D$. Let $u$ be a smooth function on $r\D$ such that 
\[
\lambda - \lambda_0 = du.
\]
Let $\{\beta_0^t\}_{t\in [0,1]}$ be a smooth path of contact forms on $r\D \times \R/\Z$ going from $\beta_0$ to $\lambda_0+ds$ and agreeing with $\lambda_0+ds$ in a neighborhood $(r\D \setminus r'\D)\times \R/\Z$ of the boundary of $r\D \times \R/\Z$, where $r'\in (0,r)$. Let $\chi$ be a smooth function on $r\D$ such that $\chi=0$ on $r'\D$ and $\chi=1$ on a neighborhood of $\partial (r\D)$. Consider the following smooth contact form on $r\D\times \R/\Z$:
\[
\beta := \beta_0 + d(\chi u).
\]
This contact form satisfies (a1) with respect to $\lambda$. The formula
\[
\beta^t :=  \beta_0^t + d(\chi u)
\]
defines a smooth path of contact forms on $r\D \times \R/\Z$ going from $\beta$ to the 1-form
\[
\lambda_0 + ds+d(\chi u),
\]
and agreeing with $\lambda+ds$ in a neighborhood of the boundary of $r\D \times \R/\Z$. The latter contact form can be joined to $\lambda + ds$ by the smooth isotopy
\[
\lambda_0 + t\, du + (1-t) d(\chi u) + ds,
\]
which consists of contact forms on $r\D \times \R/\Z$ agreeing with $\lambda+ds$ in a neighborhood of the boundary of $r\D \times \R/\Z$. This proves that $\beta$ satisfies (a2) with respect to $\lambda$.

The volume form induced by $\beta$ is
\[
\beta \wedge d\beta = \beta_0 \wedge d\beta_0 + d(\chi u) \wedge d\beta_0.
\]
The 3-form $d(\chi u) \wedge d\beta_0$ vanishes identically, because $d(\chi u)$ is supported in the region in which $d\beta_0=d\lambda_0$ and the contractions of this 1-form and this 2-form along $\partial_s$ are both zero.
We deduce that $\beta\wedge d\beta = \beta_0 \wedge d\beta_0$ on the whole $r\D\times \R/\Z$, and the fact that $\beta_0$ satisfies (a4) implies that also $\beta$ does.

The fact that $\beta$ differs from $\beta_0$ by an exact 1-form which vanishes along the direction of $R_{\beta_0}$ implies that the Reeb vector field of $\beta$ coincides with the one of $\beta_0$. Therefore, the fact that $\beta_0$ satisfies (a3) implies that also $\beta$ does.
\end{proof}

\section{Proof of the main theorem}
\label{theproof}

We are now ready to prove the theorem stated in the introduction. 

Let $\xi$ be a contact structure on the closed three-manifold $M$. Let $S$ be the smoothly embedded compact surface $S\subset M$ given by Proposition  \ref{prop_nice_contact_form} and let $\epsilon<1$ be a fixed positive number. By Proposition \ref{prop_nice_contact_form} we can find a contact form $\tilde{\alpha}$ on $M$ such that $\ker \tilde{\alpha} = \xi$, $S$ is a global surface of section for the Reeb flow of $\tilde{\alpha}$ and statements (i)-(iv) hold, where $\alpha$ is renamed as $\tilde{\alpha}$. In the following, $U$, $\varphi$ and $\tau$ are the objects appearing in this proposition. Statements labeled by a roman number refer to the statements of this proposition.

Denote by $C_1,\dots,C_{\ell}$ the circles forming the boundary of $S$. Each connected component of $U$ is a half-open annulus $A_j$ and $d\tilde{\alpha}$ restricts to an area form on the interior of each $A_j$. Let $a_j>0$ be the $d\tilde{\alpha}$-area of each $A_j$. By (i) and (iii) we have
\begin{equation}
\label{area} 
\ell = \int_S d\tilde{\alpha} = \int_{S\setminus U} d\tilde{\alpha}+ \sum_{j=1}^{\ell} \int_{A_j} d\tilde{\alpha} < \epsilon + \sum_{j=1}^{\ell} a_j.
\end{equation}
For any $j\in \{1,\dots,\ell\}$, let $r_j>0$ be such that
\[
\pi r_j^2 = (1-\epsilon) a_j,
\]
and let 
\[
\psi_j: r_j \mathbb{D} \hookrightarrow A_j \setminus \partial S
\]
be a smooth embedding such that 
\[
\psi_j^*(d\tilde{\alpha}) = dx\wedge dy.
\]
Such an area-preserving embedding exists because the Euclidean area of $r_j \D$ is smaller than the $d\tilde{\alpha}$-area of $A_j$. Using the fact that $\varphi$ is the identity and $\tau=1$ on each $A_j$, we can use the Reeb flow of $\tilde{\alpha}$ to lift each $\psi_j$ to an embedding
\[
\Psi_j: r_j \D \times \R/\Z \hookrightarrow M
\]
such that
\begin{equation}
\label{lift}
\Psi_j(\cdot,0) = \psi_j \qquad \mbox{and} \qquad \Psi_j^*(R_{\tilde{\alpha}}) = \partial_s,
\end{equation}
where $s$ denoted the coordinate on $\R/\Z$.
We claim that
\begin{equation}
\label{claim}
\Psi_j^*(\tilde{\alpha}) = \lambda_j + ds
\end{equation}
for some primitive $\lambda_j$ of $dx\wedge dy$ on $r_j\D$. Indeed, set for simplicity $\tilde{\alpha}_j:= \Psi_j^*(\tilde{\alpha})$. The the second identity in (\ref{lift}) implies that the Reeb vector field of $\tilde{\alpha}_j$ is $\partial_s$. Since any contact form is invariant by its Reeb flow, $\tilde{\alpha}_j$ is invariant under the translations $(x,y,s) \mapsto (x,y,s+t)$ and hence has the form
\[
\tilde{\alpha}_j = \lambda_j + u_j(x,y)\, ds,
\]
where $\lambda_j$ is a 1-form on $r_j \D$ and $u_j$ is a function on $r_j\D$. Then the condition $\tilde{\alpha}_j(\partial_s)=1$ implies that $u_j=1$. Finally, the first identity in (\ref{lift}) implies that 
\[
d\lambda_j = d\tilde{\alpha}_j|_{r_j \D \times \{0\}} = \psi_j^*(d\tilde{\alpha}) = dx\wedge dy,
\]
which concludes the proof of the claim.

Denote by $W_j$ the image of the embedding $\Psi_j$. Its volume with respect to $\tilde{\alpha} \wedge d\tilde{\alpha}$ is
\[
\begin{split}
\mathrm{vol}\,(  W_j , \tilde{\alpha}\wedge d\tilde{\alpha}) &=\mathrm{vol}\,( r_j \D \times \R/\Z, \tilde{\alpha}_j \wedge d\tilde{\alpha}_j ) = \mathrm{vol}\,( r_j \D \times \R/\Z, d\lambda_j\wedge ds) \\ &=  \mathrm{vol}\,( r_j \D \times \R/\Z,dx\wedge dy\wedge ds) = \pi r_j^2 = (1-\epsilon)a_j.
\end{split}
\]
By (\ref{vol-tau}), (i) and (iv), the total volume of $M$ with respect to $\tilde{\alpha}\wedge d\tilde{\alpha}$ has the upper bound
\[
\mathrm{vol}\,(M,\tilde{\alpha}\wedge d\tilde{\alpha}) = \int_S \tau\, d\tilde{\alpha} \leq (1+\epsilon) \int_S  d\tilde{\alpha} = (1+\epsilon) \ell.
\]
Since the sets $W_j$ are pairwise disjoint, being the saturations by the flow of pairwise disjoint sets on a global surface of section, the volume of the complement of their union has the upper bound
\begin{equation}
\label{volume_compl}
\begin{split}
\mathrm{vol}\,\Bigl(  M \setminus \bigcup_{j=1}^{\ell} &W_j , \tilde{\alpha}\wedge d\tilde{\alpha}\Bigr) = \mathrm{vol}\,(M,\tilde{\alpha}\wedge d\tilde{\alpha}) - \sum_{j=1}^{\ell} \mathrm{vol}\,(W_j,\tilde{\alpha}\wedge d\tilde{\alpha}) \\
&\leq (1+\epsilon) \ell - (1-\epsilon) \sum_{j=1}^{\ell} a_j <  (1+\epsilon) \ell - (1-\epsilon)(\ell - \epsilon) \\ &= \epsilon(2\ell+1) - \epsilon^2 < \epsilon(2\ell+1),
\end{split}
\end{equation}
where the second inequality follows from (\ref{area}). 

Thanks to Proposition \ref{prop0}, on every solid torus $r_j \D\times \R/\Z$ there is a contact form $\beta_j$ which agrees with $\lambda_j+ds$ near the boundary, has total volume less than $\epsilon$ and is such that all its closed Reeb orbits have period at least 1. Moreover, $\beta_j$ is smoothly isotopic to $\lambda_j+ds$ through contact forms which agree with $\lambda_j+ds$ near the boundary.

Let $\alpha$ be the contact form on $M$ which coincides with $\tilde{\alpha}$ outside of the union of the $W_j$ and on each $W_j$ is given by the push-forward of $\beta_j$ by the embedding $\Psi_j$. This form is indeed smooth due to (\ref{claim}). It is smoothly isotopic to the contact form $\tilde{\alpha}$, and hence by Gray stability its kernel is diffeomorphic to the given structure $\xi=\ker \tilde{\alpha}$. By (\ref{volume_compl}) and the fact that the volume of each $r_j \D\times \R/\Z$ with respect to $\beta_j\wedge d\beta_j$ is smaller than $\epsilon$, the volume of $M$ with respect to $\alpha\wedge d\alpha$ has the upper bound
\begin{equation}
\label{volume}
\begin{split}
\mathrm{vol}\,(M,\alpha\wedge d\alpha) &= \mathrm{vol}\,\Bigl(  M \setminus \bigcup_{j=1}^{\ell} W_j , \tilde{\alpha}\wedge d\tilde{\alpha}\Bigr) + \sum_{j=1}^{\ell} \mathrm{vol}\,(r_j \D\times \R/\Z,\beta_j\wedge d\beta_j) \\ &< \epsilon(2\ell+1) + \epsilon \ell = \epsilon(3\ell+1).
\end{split}
\end{equation}

The sets $W_j$ are invariant for the Reeb flow of $\alpha$. The closed orbits which are contained in the $W_j$'s have period at least 1, thanks to the corresponding property of the $\beta_j$'s. The components of the boundary of $S$ are closed orbits of period 1, thanks to statement (i) in Proposition \ref{prop_nice_contact_form}. Statement (iv) in the same proposition tells us that all the other closed orbits of the Reeb flow of $\alpha$ have period larger than $1-\epsilon$. We conclude that
\[
T_{\min}(\alpha) > 1-\epsilon.
\]
Together with (\ref{volume}), we deduce that the systolic ratio of $\alpha$ has the lower bound
\[
\rho_{\mathrm{sys}}(\alpha) = \frac{T_{\min}(\alpha)^2}{\mathrm{vol}\,(M,\alpha\wedge d\alpha)} > \frac{(1-\epsilon)^2}{\epsilon(3\ell+1)}.
\]
Since the latter quantity tends to $+\infty$ for $\epsilon\rightarrow 0$, the systolic ratio of a contact form inducing the contact structure $\xi$ on $M$ can be made arbitrarily large. This concludes the proof of the theorem stated in the introduction.


\providecommand{\bysame}{\leavevmode\hbox to3em{\hrulefill}\thinspace}
\providecommand{\MR}{\relax\ifhmode\unskip\space\fi MR }
\providecommand{\MRhref}[2]{%
  \href{http://www.ams.org/mathscinet-getitem?mr=#1}{#2}
}
\providecommand{\href}[2]{#2}

\end{document}